\documentclass[11pt,a4paper]{article}%
\usepackage[centertags]{amsmath}
\usepackage{amsfonts}
\usepackage{amssymb}
\usepackage{amsthm}
\usepackage{epsfig}
\usepackage{setspace}
\usepackage{ae}
\usepackage{eucal}
\usepackage[usenames]{color}%
\usepackage{graphicx}

\theoremstyle{plain}
\newtheorem{thm}{Theorem}[section]

\newtheorem{cor}[thm]{Corollary}
\newtheorem{prop}[thm]{Proposition}
\theoremstyle{definition}

\theoremstyle{remark}

\newtheorem{rem}[thm]{Remark}

\begin{document}

\title{How to compute the length of a geodesic on a Riemannian manifold with small error in arbitrary Sobolev norms}
\author{J\"org Kampen $^{1}$ }
\maketitle

\begin{abstract}
We compute the length of geodesics on a Riemannian manifold by regular polynomial interpolation of the global solution of the eikonal equation related to the line element $ds^2=g_{ij}dx^idx^j$ of the manifold. Our algorithm approximates the length functional in arbitrarily strong Sobolev norms. Error estimates are obtained where the geometric information is used.  It is pointed out how the algorithm can be used to get accurate approximation of solutions of parabolic partial differential equations leading obvious applications to finance and physics.
\end{abstract}

\footnotetext[1]{Weierstrass Institute for Applied Analysis and Stochastics,
Mohrenstr. 39, 10117 Berlin, Germany.
\texttt{{kampen@wias-berlin.de}}.}

\section{Introduction}
Let $(M,g)$ is a Riemannian manifold, i.e. a differentiable $n$-dimensional manifold with a function $g$, which defines for all $p\in M$ a positive definite symmetric bilinear form
\begin{equation}
g_p:T_pM\times T_pM\rightarrow {\mathbb R}
\end{equation}
such that for any given vector fields $X,Y\in X(M)$ the map
\begin{equation}
g(X,Y): M\rightarrow {\mathbb R},~p\rightarrow g(X,Y)(p):=g_p(X_p,Y_p)
\end{equation}
is differentiable. The Riemannian metric $g$ allows to define a metric $d_M$ on $M$ via the length of curves
\begin{equation}
d_M(x,y):=\inf_{\mbox{$\gamma$ diff.}}\left\lbrace L(\gamma)|\gamma :[0,1]\rightarrow M, \gamma(0)=x, \gamma(1)=y\right\rbrace, 
\end{equation}
with
\begin{equation}
    L(\gamma)=\int_0^1 \sqrt{g_{\gamma(t)}(\dot \gamma(t),\dot \gamma(t))} \,\mathrm dt. 
\end{equation}
With this definition any connected Riemannian manifold becomes a metric space, and 
it is well known that for any compact Riemannian manifold any two points $x,y\in M$ can be connected by a geodesic whose length is $d_M(x,y)$. 
If $\nabla$ denotes the Levi-Civita connection, then a geodesic $\gamma$ is characterized by the equation
\begin{equation}
\nabla_{\dot\gamma}\dot\gamma=0,
\end{equation}
which becomes (in terms of the coordinates of the values of the curve $\gamma$)
\begin{equation}\label{geo1}
    \frac{d^2x^\lambda }{dt^2} + \Gamma^{\lambda}_{~\mu \nu }\frac{dx^\mu }{dt}\frac{dx^\nu }{dt} = 0\ , 
\end{equation}
where the well-known Christoffel symbols are 
\begin{equation}\label{geo2}
    \Gamma^\kappa_{\; \mu \nu}=\frac{1}{2}g^{\kappa \rho} \left( \partial_\mu g_{\nu \rho}+\partial_\nu g_{\mu \rho}-\partial_\rho g_{\mu \nu} \right).
\end{equation} 
This is an $n$-dimensional nonlinear ordinary differential equation with values in ${\mathbb R}^n$ which is difficult to compute numerically in general (note the quadratic terms). For computing the length of a geodesic it is easier to compute the solution of a  eikonal equation of the form
\begin{equation}\label{eik}
d^2=\frac{1}{4}\sum_{ij}a_{ij}(x)d^2_{x_i}d^2_{x_j}
\end{equation}
(boundary conditions considered later), where $x\rightarrow a_{ij}(x)$ are functions such that at each $x\in{\mathbb R}^n$ the matrix $(a_{ij}(x))$ is the inverse of the positive matrix $(g_{ij}(x))$ at each point $x$. Here $f_{x_i}:=\frac{\partial f}{\partial x_i}$ denotes the derivative of $f$ with respect to the variable $x_i$.
In general we shall write $\partial^{\alpha}f$, $\partial^{\alpha}_xf$ or $\frac{\partial}{\partial x^{\alpha}}f$ for the multivariate derivative with multiindex $\alpha=(\alpha_1,\cdots ,\alpha_n)$.
The connection between the length of a geodesic which is given in local coordinates as in \eqref{geo1}, \eqref{geo2} and the length function $d^2$ defined by equation \eqref{eik} is considered in section 2.
This way the problem of finding the length of a geodesic is reduced to solving a nonlinear first-order partial differential equation in some domain of Euclidean space.

The computation of $d^2$ is still far from trivial, however. Even if the data $g_{ij}$ are analytic functions, power series expansion typically lead to power series solutions for $d^2$ with small radius of convergence. 
Hence, the question is how we can approximate the function $d^2$ globally. Moreover, for some applications such as the accurate computation of diffusions we need the approximation of $d^2$ in strong norms (Sobolev norms of form
$H^{s,p}$ for possibly any positive real $s$. For that matter recall that $H^{0,p}\left( {\mathbb R}^n\right) =L^p\left( {\mathbb R}^n\right)$ and that for any $s\in {\mathbb R}$ we may define $H^{s,p}$ to be the set of all tempered distributions $\phi\in {\cal S}'$ such that $I_{-s}\phi$ is a function in $L^p\left( {\mathbb R}^n\right)$, where
$I_s$ is the pseudo-differential operator with symbol $\sigma_s(\xi)=\left(1+|\xi|^2\right)^{-\frac{s}{2}}$, i.e.
\begin{equation}
 I_s\phi ={\cal F}^{-1}\sigma_s{\cal F} \phi,~~\phi\in {\cal S}',
\end{equation}
${\cal F}$ denoting the Fourier transform. The goal of the present paper can then be formulated as follows: find for each $\epsilon >0$ and each real $s,p$ $(p\geq 1)$ an approximative solution $q^2_{s,p}$ to \eqref{eik} such that  
\begin{equation}
\|d^2-q_{s,p}^2\|_{s,p}\leq \epsilon .
\end{equation}
We shall call $q_{s,p}^2$ an $H^{s,p}$ approximation to $d^ 2$ for reasons which will become apparent later.
Let us motivate this ambitious task by looking at a specific application.
There are a lot of applications for computations of the length of a geodesic, where applications to computations in general relativity are only one domain. Another important example is the leading term of the expansion of the fundamental solution of linear parabolic solutions (with variable coefficients). Varadhan showed that
the fundamental solution of the diffusion equation
\begin{equation}\label{PPDE} 
\begin{array}{l}
\frac{\partial u}{\partial t}=\frac{1}{2}\sum_{i,j}a_{ij}\frac{\partial^2 u}{\partial x_i\partial x_j}+
\sum_i b_i\frac{\partial u}{\partial x_i},
\end{array}
\end{equation}
(where the diffusion coefficients $a_{ij}$ and the first order coefficients $b_i$ in \eqref{PPDE}
depend on the spatial variable $x$ only) is connected to the length $d$ of the geodesic with respect to the line element $ds^2=\sum_{ij}a^{ij}dx_idx_j$ ($a^{ij}$ being the inverse of $a_{ij}$) via the relation
\begin{equation}
d^2(x,y)=\lim_{t\downarrow 0}t\ln p(t,x,y).
\end{equation}
\begin{rem}
Solving equation \eqref{eik} we can assume that the matrix-valued function $x\rightarrow (a_{ij}(x))$ is symmetric, i.e.
$a_{ij}(x)=a_{ji}(x)$ for all $1\leq i,j\leq n$.  This is because
\begin{equation}
\begin{array}{ll}
d^ 2(x,y)=&\frac{1}{4}\sum_{ij}a_{ij}d^ 2_{x_i}d^ 2_{x_j}=\frac{1}{4}\sum_{ij}\frac{1}{2}\left( a_{ij}+a_{ji}\right) d^ 2_{x_i}d^ 2_{x_j}\\
\\
&+\frac{1}{4}\sum_{ij}\frac{1}{2}\left( a_{ij}-a_{ji}\right) d^ 2_{x_i}d^ 2_{x_j}=\frac{1}{4}\sum_{ij}\frac{1}{2}\left( a_{ij}+a_{ji}\right) d^ 2_{x_i}d^ 2_{x_j},
\end{array}
\end{equation}
so we can always substitute the matrix $a_{ij}$ by its symmetrization $\frac{1}{2}\left( a_{ij}+a_{ji}\right)$ without affecting the solution $d^ 2$. 
\end{rem}
In \cite{Ka} we have seen that for $C^{\infty}$ coefficient functions $x\rightarrow a_{ij}(x)$ and $x\rightarrow b_i(x)$ and if some boundedness conditions of the derivatives are satisfied the fundamental solution has the pointwise valid form
\begin{equation}\label{WKBrep}
p(t,x,y)=\frac{1}{\sqrt{2\pi t}^n}\exp\left(-\frac{d^2(x,y)}{2 t}+\sum_{k= 0}^{\infty}c_k(x,y)t^k\right), 
\end{equation}
where the functions $x\rightarrow c_k(x,y),~k\geq 0$ are solutions of recursively defined linear first order equations for each $y$. These equations can be solved by methods of characteristics or approximated by regular polynomial interpolation methods outlined in \cite{Ka2}. In the computation of the WKB-coefficients $d^2$ and $c_k,~k\geq 0$ the recursive relations for $c_{k+1}$ involve second order derivatives of $c_k$, and therefore implicitly derivatives of order $2k$ of the squared metric $d^2$. Hence it is of great interest to compute not only $d^2$ but also its derivatives up to a given order with high accuracy. The present work shows how his can be accomplished. In Section 2 we recapture some facts about the connection of the geodesic equation \eqref{geo1}, \eqref{geo2} and equation \eqref{eik}, and prove global existence, regularity and uniqueness of the latter (family) of equation(s) leading us to theorem 2.3. Then in Section 3 we provide further analysis of the family of eikonal equations which lead us to local representations of the solution. 
In Section 4 we construct first a weak approximation of the solution (in $L^p$ sense), and then extend this to a recursive construction of an $H^{s,p}$-approximation.  In Section 5 we provide error estimates by using geometric information. Section 6 points out how the method may be applied for accurate approximation of diffusions, and we finish with a conclusion in Section 7.

\section{Global existence and regularity of the squared Riemannian distance $d^2$}
We shall only sketch the connection between geodesics and the eikonal equation \eqref{eik}. It is almost standard,
and details can be found in \cite{Ka} and \cite{Jo}. Our interest here is that the eikonal equation together with careful chosen boundary conditions has a global and unique solution. We shall have two different arguments for uniqueness: one is via uniqueness of an associated diffusion and WKB-representations (or, alternatively, Varadhan's result, cf. \cite{V}), but we will have the same insight from an other point of view when we look at local representations of the solution in the next Section.
We consider Riemannian manifolds where any two points can be connected by a minimal geodesic. For our purposes it is sufficient to consider manifolds which are geodesically complete.
Recall that a Riemannian manifold $M$ is geodesically complete if for all
$p\in M$ the exponential map $\exp_p: T_pM\rightarrow M$ is defined globally on $T_pM$. Here, $T_pM$ denotes the tangential space of the manifold $M$ at $p\in M$. The Hopf-Rinow theorem provides conditions for Riemannian manifolds to be geodesically complete. Especially we have
\begin{thm}
For a Riemannian manifold $M$ the following statements are equivalent:
\begin{itemize}
\item $M$ is complete as a metric space.

\item The closed and bounded sets of $M$ are compact.

\item $M$ is geodesically complete.
\end{itemize}
Each of these equivalent statements implies that geodesics are curves of shortest length. Moreover,
if $M$ is geodesically complete, then any two points of $M$ can be joined by a minimal geodesic.  
\end{thm}
The connection between the arclength and equation \eqref{eik} can be established as follows. First equations for minimal geodesics are obtained from variation of the length functional. Second Hamilton-Jacobi calculus shows that the length functional satisfies the eikonal equation \eqref{eik}. Since this is known we only sketch the main steps for convenience of the reader.
Setting the variation of the length functional to zero we get
\begin{equation}
L\frac{d}{dr}\left(\frac{1}{L}2g_{ij}\dot x^i \right)+g_{ij,k} \dot x^i \dot x^j=0
\end{equation}
with $L\equiv \sqrt{g_{ij}(x(r))\dot x^i\dot x^j}$ and where we use Einstein summation. Parameterizing by arclength, i.e. setting $L\equiv 1$ (or $r=s$) we get
\begin{equation}
2g_{ij}\ddot x^i +2g_{ij,l}\dot x^l\dot x^i +g_{ij,k} \dot x^i \dot x^j=0
\end{equation}
which, upon multiplcation by $g^{mj}$ (entries of inverse of $(g_{mj})$) and rearranging becomes the geodesic equation \eqref{geo1},\eqref{geo2}. 
In order to show on the other hand that the squared length functional satisfies \eqref{eik} we may consider the length functional 
\begin{equation}
l(r,x,s,y)=\int_r^s L\left( x(u),\dot x(u)\right) du
\end{equation}
and invoke Hamilton-Jacobi calculus. This is done by
introducing the variables $p_i=L_{\dot x^i}$, and the associated Hamiltonian defined by
\begin{equation}
H(x,p)=\dot x^ip_i-L(x,\dot x).
\end{equation}
(here and henceforth we use Einstein summation if convenient).
Then we may write
$$x(t)\equiv x(t;r,x,s,y) \mbox{ and } p(t)\equiv p(t;r,x,s,y),$$
where $x(r;r,x,s,y)=x$ and $x(s;r,x,s,y)=y$. 
and compute
\begin{equation}\label{1ls}
l_s=-H(x(s),p(s)).
\end{equation}
Then we may connect $p$ to $l_{y^k}$ by computing
\begin{equation}\label{ly}
\begin{array}{ll}
l_{y^k}=\int_r^s \left( \frac{\partial \dot x^i}{\partial y^k}p_i+\dot x^i\frac{\partial p_i}{\partial y^k}
-H_{x^i}\frac{\partial x^i}{\partial y^k}-H_{p_i}\frac{\partial p_i}{\partial y^k}\right) dt\\
\\
\int_r^s \stackrel{{\bf \cdot}}{\left( \frac{\partial x_i}{\partial y^k}p_i\right)}dt=\frac{\partial x_i}{\partial y^k}p_i \Big|^s_r=p_k(s;r,x,s,y).
\end{array}
\end{equation}
by invoking the canonical system of equations. This leads to 
\begin{equation}
\frac{\partial l}{\partial s}+\sum_{ij}g^{ij}\frac{\partial l}{\partial y_i}\frac{\partial l}{\partial y_j}=0,
\end{equation}
and a similar equation with respect to the variables $x$. Then we get the equations for $l^2$ and $d^2$, i. e. the equations \eqref{lg} and \eqref{lf} below.

Recall that a minimal geodesic is a global distance minimizing geodesic. This minimal geodesic which connects $x$ and $y$ characterizes the Riemannian distance $d(x,y)$ in an obvious way. Moreover smoothness of $(x,y)\rightarrow d(x,y)$ for smooth diffusion and drift coefficients $a_{ij},b_i$ follows from
the following fact about ordinary differential equations. 
\begin{thm}Let $F:{\mathbb R}^n\times {\mathbb R}^n\rightarrow {\mathbb R}$ be a smooth map. Consider the differential system
\begin{equation}
\frac{d^2x}{dt^2}=F\left( x,\frac{dx}{dt}\right), 
\end{equation}
where $x$ is a map $I\subset {\mathbb R}\rightarrow {\mathbb R}^n$. Then for each point $(x_0,y_0)$
there exists a neighborhood $U\times V$ of this point and $\epsilon >0$ such that for $(x,v)\in U\times V$ equation (2.69) has a unique solution $x_v:]-\epsilon,\epsilon[\rightarrow {\mathbb R}^n$ 
with initial conditions $x_v(0)=x$ and $x_v'(0)=v$. Moreover, the map $X:U\times V\times ]-\epsilon,\epsilon[\rightarrow {\mathbb R}^n$ defined by $(t,x,v)\rightarrow X(t,x,v):=x_v(t)$ is smooth.
\end{thm}
Finally we get
\begin{thm}
Let $\Omega \subseteq {\mathbb R}^n$ be some domain. The function $d^2:\Omega\times \Omega \subseteq {\mathbb R}^n\times {\mathbb R}^n\rightarrow {\mathbb R}_+$ (the leading order term of the WKB-expansion of a parabolic equation with diffusion coefficients $a_{ij}$) is the unique function which satisfies the equations
\begin{equation}\label{lf}
d^2=\frac{1}{4}\sum_{ij}d_{x_i}^2a_{ij}d_{x_j}^2,
\end{equation}
\begin{equation}\label{lg}
d^2=\frac{1}{4}\sum_{ij}d_{y_i}^2a_{ij}d_{y_j}^2
\end{equation}
for all $x,y\in {\mathbb R}^n$ and with the boundary condition  
\begin{equation}\label{lh}
d(x,y)=0 \mbox{ iff  $x=y$ for all $x,y\in {\mathbb R}^n$.}
\end{equation}
Moreover, the squareroot $d$ is the Riemannian distance induced by
\begin{equation}
\begin{array}{ll}
d(x,y):=\inf{\Bigg\{}\int_a^b&\sqrt{a^{ij}(\gamma)\stackrel{.}{\gamma}^i\stackrel{.}{\gamma}^j}dt|\gamma:[a,b]\rightarrow {\mathbb R}^n \mbox{ is piecewise } \\
&\mbox{ smooth with $\gamma(a)=x$ and $\gamma(b)=y$}{\Bigg \}}.
\end{array}
\end{equation}
The function $d^2$ is a $C^{\infty}$-function with respect to both variables. 
\end{thm}
\begin{proof}
The variation of the length functional leads to the geodesic equation. On the other hand, Hamilton-Jacobi calculus leads us to the fact that the squared length functional $d^2$ satisfies the equation \eqref{eik}. It is clear that the squared length functional satisfies both equations \eqref{lf} and \eqref{lg} below. Moreover, it is clear that the squared length functional satisfies the initial condition \eqref{lh}. 
Uniqueness is a bit more subtle. In \cite{V} Varadhan showed that
\begin{equation}\label{V}
d^2(x,y)=\lim_{t\downarrow 0}2t\ln p(t,x,y),
\end{equation}
where $p$ is the fundamental solution of a scalar parabolic equation with diffusion coefficient function $x\rightarrow a_{ij}(x)$. Since $p$ is unique for a strictly parabolic equation $d^2$ is uniquely determined 
by the equation \eqref{V}. On the other hand one knows that for small $t>0$ $\ln p$ has for $C^{\infty}$ coefficients a representation of type \eqref{WKBrep} is valid (cf. \cite{Gi,Ka}). Plugging this into the correspondend parabolic equation leads to the eikonal equation \eqref{eik} which is, hence, satisfied by $d^2$. Moreover we know by $V$ and the fact that the squareroot of $d^2$ is a metric. Hence $d(x,y)=0$ if and only if $x=y$, and the same holds for $d^2$. Hence, we conclude that the global solution $d^2$ of the system of equations \eqref{eik},\eqref{lh} and \eqref{lh} is unique. Moreover, from the preceding theorem we can conclude that the function $(x,y)\rightarrow d^2(x,y)$ is also smooth with respect to both variables.
\end{proof}

\section{Further analysis of the equation for the squared metric $d^2$}

Next we observe that the local representation of the solution of the equations \eqref{lf}, \eqref{lg} with the boundary condition \eqref{lh} has a local representation which starts with the quadratic terms. This will be used in the construction of a global approximation. The analysis presented here gives us two other insights. First, a powere series ansatz leads atmost to local and not to global solutions. Even if there is a local power series representation of the solution at each point of the domain, we do not know how a global solution can be constructed from this information, because we do not know the location of the geodesic the length of which we want to compute. If we knew, then computing the length would be a rather trivial task. Even the derivatives of the length functional would be better computed from the explicit geodesic. However, as we mentioned the nonlinear ordinary differential equation describing the geodesic is harder to solve in general than the eikonal equation. Second, we shall see from an different point of view why the boundary condition \eqref{lh} leads to uniqueness of solutions $(x,y)\rightarrow d^2(x,y)$ of the system \eqref{lf}, \eqref{lg}, and \eqref{lh}.
We have \newpage
\begin{cor}
The local representation $d^2$ satisfying the equations \eqref{lf}, \eqref{lg}, together with the boundary condition \eqref{lh} is of the form
\begin{equation}\label{loc}
\begin{array}{ll}
d^2(x,y)&=\sum_{ij}a^{ij}(y)\Delta x^i\Delta x^j + \sum_{|\alpha < M }\frac{d^2_{\alpha}(y)}{\alpha!}\Delta x^{\alpha}\\
\\
&+\sum_{|\gamma|=M}\int_0^1(1-\theta)^{M-1}\frac{\Delta x^{\gamma}}{\gamma !}\partial^{\gamma}d^2(y+\theta \Delta x,y)d\theta.
\end{array}
\end{equation}
The coefficients $d_{\alpha}(y)$ are uniquely determined by a recursion obtained from the equations \eqref{lf}, \eqref{lg}. In coordinates with second order normal form, i.e. where $d^2$ is $\sum_{ij}\lambda_i(y)\Delta x^i\Delta x^j$ with $\lambda_i(y) , 1\leq i \leq n$ is the spectrum of $(a^{ij}(y))$,
the multiindex recursion is
\begin{equation}\label{rec}
\begin{array}{ll}
 d^2_{\beta}(y)=&\frac{1}{\left(1-\sum_i\beta_i\right)}{\Bigg (} \sum_i  \left(  \lambda^i_0\right)^2\frac{\lambda_i^{\beta\dot -2_i}}{(\beta\dot -2_i)!}1_{\left\lbrace \beta_i\geq 2\right\rbrace }\\
\\
&+\sum_i \sum_{|\alpha|\geq 1,|\gamma | \geq 3, \alpha+\gamma =\beta}\frac{\lambda_i^{\alpha}}{\alpha!}  \lambda^i_0 d^2_{\gamma}(y)\gamma_i\\
\\
&+ \sum_i  \sum_{\alpha\geq 0,|\delta|\geq 3,|\gamma| \geq 3,\alpha+\gamma+\delta\dot-2_i=\beta}\frac{\lambda_i^{\alpha}}{\alpha!}  \delta_i \gamma_i d^2_{\delta}(y)d^2_{\gamma}(y){\Bigg )}.
\end{array}
\end{equation}
This confirms uniqueness. (Note that there is no loss of generality if we choose the normal coordinates for the second order terms). In general the solution is not globally analytic in the sense that $d^2$ is not representable by a globally converging power series.
\end{cor}

\begin{proof}
A smooth solution $d^2$ of the eikonal equation has the representation
\begin{equation}
\begin{array}{ll}
d^2(x,y)&=d(y,y) + \nabla d(y,y)\cdot (x-y)\\
\\
&+\sum_{|\gamma|=2}\int_0^1(1-\theta)^{1}\frac{\Delta x^{\gamma}}{\gamma !}\partial^{\gamma}d^2(x+\theta \Delta x,y)d\theta.
\end{array}
\end{equation}
We abbreviate  $R(x,y)=\sum_{|\gamma|=2}\int_0^1(1-\theta)^{1}\frac{\Delta x^{\gamma}}{\gamma !}\partial^{\gamma}d^2(x+\theta \Delta x,y)d\theta.$
Since $d(y,y)=0$ we have
\begin{equation}
\begin{array}{ll}
d^2(x,y)=\nabla d^2(y,y)\cdot (x-y)+R(x,y)
\end{array}
\end{equation}
The 'only if'-condition of the boundary condition leads to $\nabla d^2(y,y)=0$. To see this assume that
$\nabla d^2(y,y)\neq 0$. Since $R(x,y)\leq C\|\Delta x\|^2$ there is a small $\Delta x$ such that 
$\nabla d^2(y,y)\cdot \mu\Delta x >C\|\Delta x\|^2$ and $\nabla d^2(y,y)\cdot (-\mu)\Delta x <-C\|\Delta x\|^2$ for some $\mu\in (0,1]$. Hence there exists some $\rho$ such that with $x':=y+(\rho\mu)\Delta x$
\begin{equation}
d^2(x',y)=\nabla d^2(y,y)\cdot (\rho\mu)\Delta x + R(x',y)=0,
\end{equation}
contradicting one part of the boundary condition $d^2(x,y)=0~\mbox{ iff }~x=y$.
Next one computes that $\sum_{ij}a^{ij}(y)\Delta x^i\Delta x^j$ satisfies the equation
\begin{equation}
d^2(x,y)=\frac{1}{4}\sum_{ij}a_{ij}(y)d^2_{x_i}d^2_{x_j},
\end{equation}
and the uniqueness of theorem 2.3. (which we established by arguing with uniqueness of related diffusions and Varadhan's result, in the Atiyah-Singer spirit of short-range analytic expansions) identifies the coefficients $a^{ij}(y)$ as the second order terms of local representations around $y$. Having obtained this the representation \eqref{loc} is just a multivariate version of Taylor's theorem. Note, however, that we do not need to invoke the uniqueness of theorem 2.3. but just consider a recursion obtained from a power series ansatz starting with second order terms. However, this would complicate the matter a bit so we take advantage that we know the second order terms of a local representation by the preceding argument. Finally we have to establish the recursion in \eqref{rec}. The recursion shows directly that the higher order coefficients $d^2_{\beta}(y)$ for $|\beta|\geq 3$ are uniquely determined. Moreover, it is clear from \eqref{rec} that in general the convergence radius of the full power series is small (if not zero). Hence in general there is no globally analytic solution
 the function $d^ 2:\Omega\times \Omega \rightarrow {\mathbb R}$ globally analytic if for each $y\in \mathbb{R}^n$
the Taylor expansion of $d^2$ at $y\in \mathbb{R}^n$ and $x\in \mathbb{R}^n$ 
equals $d^2$ globally, i.e.
\begin{equation}\label{geolgth}
d^ 2(x,y)=\sum_{\alpha}\frac{\partial_{\alpha}d^2(y)}{\alpha !}(x-y)^ {\alpha}~~\mbox{forall}~x,y\in \mathbb{R}^n. 
\end{equation}
Invoking the implicit function theorem equation \eqref{eik} is equivalent to
\begin{equation}\label{eik2}
d^2=\frac{1}{4}\sum_{i}\lambda_{i}(x)d^2_{x_i}d^2_{x_i},
\end{equation}
where $\lambda_i(x), 1\leq i\leq n$ is the spectrum of the positive $(a_{ij}(x))$. Since $d^2_{x_i}=2d d_{x_i}$ this is equivalent to
\begin{equation}\label{eik3}
1=\sum_{i}\lambda_{i}(x)d_{x_i}d_{x_i}
\end{equation}
The latter equation is easier but there is no Taylor expansion around $y$ as can be seen in the case of constant coefficients (and hence constant eigenvalues $\lambda$), where the solution is
\begin{equation}
d(x,y)=\sqrt{\sum_{i=1}^n \frac{\Delta x_i^2}{\lambda_i}}
\end{equation}
\begin{rem}
We use equation \eqref{eik2} mainly for the theoretical purposes of this corollary. In general it cannot be in general used for numerical purposes since this would imply that we have an efficient procedure to compute the eigenvalue functions of a space dependent matrix. Since we are looking for high precision in this paper, this is not possible in general. An exception is the case of dimension $n=2$ where we have
\begin{equation}
\lambda_{1,2}(x)=\frac{\mbox{tr}(A)(x)}{2}\pm\sqrt{\left( \frac{\mbox{tr}(A)(x)}{2}\right) ^2-\mbox{det}(A)(x)}
\end{equation}
where $A(x)=(a_{ij}(x))$.
\end{rem}
Next we plug in the power series expansion
\begin{equation}
d^2(x,y)=\sum_{i=1}^n\lambda^i_0\Delta x_i^2 +\sum_{|\beta |\geq 3 } d^2_{\beta}(y) \Delta x^{\beta}
\end{equation}
We have 
\begin{equation}
d^2_{x_i}=2\lambda_0^i(y)\Delta x_i +\sum_{|\beta |\geq 3 } d^2_{\beta}(y) \beta_i \Delta x^{\beta \dot -1_i},
\end{equation}
where for any multiindex $\beta$ we define
\begin{equation}
\beta\dot-1_i=(\beta_1,\cdots,\beta_i,\cdots \beta_n)\dot-1_i:=\left\lbrace \begin{array}{ll} (\beta_1,\cdots,\beta_i-1,\cdots \beta_n) \mbox{ if } \beta_i\geq 1\\ (\beta_1,\cdots,0,\cdots \beta_n)~~\mbox{ else }  \end{array}\right.
\end{equation}
The term $\beta -2_i$ is defined analogously.
Plugging in the power series ansatz and using the relation $\lambda^0_i\left(  \lambda^i_0\right)^2=\lambda^0_i$, this leads to

\begin{equation}
\begin{array}{ll}
&\left( \sum_{|\beta |\geq 3} d^2_{\beta}(y)\Delta x^{\beta}\right) \left(1-\sum_i\beta_i \lambda_0^i\lambda^0_i\right) \\
\\
=&\left( \sum_{|\beta |\geq 3} d^2_{\beta}(y)\Delta x^{\beta}\right) \left(1-\sum_i\beta_i\right)\\ 
\\
=&\left( \sum_i  \sum_{|\alpha|\geq 1}\frac{\lambda_i^{\alpha}}{\alpha!}\Delta x^{\alpha}\right)\left(  \lambda^i_0\right)^2\Delta x_i^2+\\
\\
&+\left(\sum_i \sum_{|\alpha|\geq 1}\frac{\lambda_i^{\alpha}}{\alpha!}\Delta x^{\alpha} \right)\left( \lambda^i_0\sum_{|\beta| \geq 3}d^2_{\beta}(y)\beta_i\Delta x^{\beta}\right) \\
\\
&+\left( \sum_i  \sum_{\alpha}\frac{\lambda_i^{\alpha}}{\alpha!}\Delta x^{\alpha}\right)\times\\
\\
&\left(  \sum_{|\beta|\geq 3,|\gamma|\geq 3}\beta_i \gamma_i d^2_{\beta}(y)d^2_{\gamma}(y)\Delta x^ {\beta\dot-1}\Delta x^ {\gamma\dot-1}\right).
\end{array}
\end{equation}
This leads to 
\begin{equation}
\begin{array}{ll}
&\sum_{|\beta |\geq 3} d^2_{\beta}(y)\Delta x^{\beta}\\ 
\\
=&\frac{1}{\left(1-\sum_i\beta_i\right)}{\Bigg (}\left( \sum_i  \sum_{|\alpha|\geq 1}\frac{\lambda_i^{\alpha}}{\alpha!}\Delta x^{\alpha}\right)\left(  \lambda^i_0\right)^2\Delta x_i^2+\\
\\
&+\left(\sum_i \sum_{|\alpha|\geq 1}\frac{\lambda_i^{\alpha}}{\alpha!}\Delta x^{\alpha} \right)\left( \lambda^i_0\sum_{|\beta| \geq 3}d^2_{\beta}(y)\beta_i\Delta x^{\beta}\right) \\
\\
&+\left( \sum_i  \sum_{\alpha}\frac{\lambda_i^{\alpha}}{\alpha!}\Delta x^{\alpha}\right)\times\\
\\
&\left(  \sum_{|\beta|\geq 3,|\gamma|\geq 3}\beta_i \gamma_i d^2_{\beta}(y)d^2_{\gamma}(y)\Delta x^ {\beta\dot-1}\Delta x^ {\gamma\dot-1}\right){\Bigg )}.
\end{array}
\end{equation}
Simplifying and renaming multiindices in order to collect for multiindices of order $\beta$ we get
\begin{equation}
\begin{array}{ll}
&\sum_{|\beta |\geq 3} d^2_{\beta}(y)\Delta x^{\beta} \\ 
\\
=&\frac{1}{\left(1-\sum_i\beta_i\right)}{\Bigg (} \sum_i  \sum_{|\alpha|\geq 1}\left(  \lambda^i_0\right)^2\frac{\lambda_i^{\alpha}}{\alpha!}\Delta x^{\alpha+2_i}+\\
\\
&+\sum_i \sum_{|\alpha|\geq 1}\sum_{|\gamma| \geq 3}\frac{\lambda_i^{\alpha}}{\alpha!}  \lambda^i_0 d^2_{\gamma}(y)\gamma_i\Delta x^{\alpha +\gamma} \\
\\
&+ \sum_i  \sum_{\alpha}\sum_{|\delta|\geq 3,|\gamma|\geq 3}\frac{\lambda_i^{\alpha}}{\alpha!}  \delta_i \gamma_i d^2_{\delta}(y)d^2_{\gamma}(y)\Delta x^ {\alpha+\gamma+\delta\dot-2_i}{\Bigg )}.
\end{array}
\end{equation}
The latter equation leads directly to \eqref{rec}.
\end{proof}

Let us draw some consequences out of our theoretical considerations. There is neither an explicit solution nor leads a power series ansatz to a global solution in general. Neither does it help to have local solutions in terms of power series.  Such representations are not sufficient for our purposes, since we are interested in a global solution for $x\rightarrow d^2(x,y)$ and do not know the intermediate points on the corresponding geodesic in order to compute the global $d^2$ by means of local power series representations.
This motivates our later construction of regular polynomial interpolation of $d^2$ as seemingly unavoidable.

\section{Regular polynomial interpolation algorithm for the Riemannian metric and its derivatives}
For the moment let us denote again an interpolation polynomial which approximates the squared Riemannian distance $d^2$ in the $L^p$-sense on some bounded domain $\Omega$ by $q_{0,p}^2$ and one that approximates the squared Riemannian distance $d^2$ in the $H^{s,p}$-sense (again on $\Omega$) by $q_{s,p}^2$. How can we check that a given polynomial is an approximation in either sense? 
The equation \eqref{eik} gives us itself a hint how an approximation $q_{s,p}^2$ of $d^2$ performs.
In order to obtain the $L^p$ error of an $L^p$ approximation $q^2_{0,p}$ of $d^2$ we may plug in the approximation
$q_{s,p}^2$ into the right side of equation \eqref{eik} and subtract the left side, i.e. we compute
\begin{equation}
\frac{1}{4}\sum_{ij}a_{ij}(x)\frac{\partial q^2_{0,p}}{\partial x_i}\frac{\partial q^2_{0,p}}{\partial x_j}-q^2_{0,p}=r_{0,p}(x),
\end{equation}
We shall see that $r_{0,p}\in O(h^3)$ locally (with $h$ the mesh size of the interpolation points) implies that
\begin{equation}
\|d(x,y)-q_{0,p}\|_{L^p(\Omega)}
\end{equation}
converges to zero as the number of interpolation points $N$ goes to infinity in such a way that the mesh size of the set of interpolation points $h$ goes to zero. Note that $q_{0,p}$ denotes the squareroot of $q_{0,p}^2$.
We call an approximation $q^2_{s,p}$ an $H^{s,p}$-approximation if it approximates not only $d^2$ in the $L^p$ sense but can be plugged in into all the derivatives of \eqref{eik} of order $m$ (i.e. multivariate derivatives $\alpha$ for $|\alpha|\leq  m$ of the eikonal equation) such that in
\begin{equation}
\frac{\partial^{\alpha}}{\partial x^{\alpha}}\left(\frac{1}{4}\sum_{ij}a_{ij}(x)\frac{\partial q^2_{0,p}}{\partial x_i}\frac{\partial q^2_{0,p}}{\partial x_j} \right)-\frac{\partial^{\alpha}}{\partial x^{\alpha}}d^2(x,y)=:r_{\alpha ,p}
\end{equation}
the right side  staisfies $r_{0,p}\in O(h^{3+m})$ locally  implies that
\begin{equation}
\|d(x,y)-q_{0,p}\|_{H^{s,p}(\Omega)}
\end{equation}
converges to zero as the number of interpolation points $N$ goes to infinity in such a way that the mesh size of the set of interpolation points $h$ goes to zero.
Accordingly, we call such $q_{0,p}^2$ ($q_{s,p}^2$) an $L^p$- ($H^{s,p}$) approximation of the boundary value problem \eqref{eik}. 
In the next subsection we construct a $L^p$-approximation and refine the construction in the following subsection in order to construct $H^{s,p}$-approximations.    
\subsection{Polynomial interpolation of eikonal equation in $L^ p$ sense}

We may write the eikonal equation \eqref{eik}
\begin{equation}\label{eik0}
d^2(x,y)=\frac{1}{4}\sum_{ij}a_{ij}d^ 2_{x_i}d^ 2_{x_j}=\frac{1}{4}\left\langle \nabla d^2, A\nabla d^2\right\rangle. 
\end{equation}
Assume that $A=(a_{ij})$ is constant. The solution \eqref{eik} with the boundary condition $d^2(x,y)=0$ iff $x=y$ is
\begin{equation}
d^2(x,y)=\left\langle \Delta x, A^ {-1}\Delta x\right\rangle,
\end{equation}
where $\Delta x=(x-y)$, and  $A^ {-1}=:(a^ {ij})$ denotes the inverse of the matrix $A$.
This is easily verified by observing that
\begin{equation}
\nabla d^ 2= 2 A^ {-1}x.
\end{equation}
Define
\begin{equation}
d^ 2_{A^ {-1}(x^j)}(x,y)=\sum_{ml} a^ {lm}(x_j)(x^ l-y^l)(x^ m-y^m),~~j=0,\cdots,N
\end{equation}
we get the first recursively defined approximation algorithm for the Riemannian distance based on $N+1$ interpolation points $x^0=y, x^1, x^2, \cdots ,x^N$.
Note that the squared distance is a function
\begin{equation}
d^ 2:\Omega \times \Omega \subseteq {\mathbb R}^n\times {\mathbb R}^n \rightarrow {\mathbb R}_+,
\end{equation}
where we define ${\mathbb R}_+:=\left\lbrace x| x\geq 0\right\rbrace $.
There are several ways to approximate the function $d^2$. 
In order to approximate this function we approximate first the function $x\rightarrow d^2(x,y)$, then the function 
$x\rightarrow d^ 2(x, x^1)$ and so on up to $x\rightarrow d^ 2(x, x^N)$.

We start with the approximation of $x\rightarrow d^2(x,y)$.
First define 
\begin{equation}
d_{00}^2(x,y)=d^ 2_{A^ {-1}(y)}(x,y)
\end{equation}
Next define
\begin{equation}
d_{10}^2(x,y)=d^ 2_{A^ {-1}(y)}(x,y)+c_{10}\Pi_{l=1}^n(x_l-y_l)^2d^ 2_{A^ {-1}(x_1)}(x,y),
\end{equation}
and determine a real number $c_{10}$ such that
\begin{equation}
d_1^2(x_1,y)=\frac{1}{4}\sum_{ij} a_{ij}(x_1) d^ 2_{1,x_i}(x_1,y)d^ 2_{1,x_j}(x_1,y),
\end{equation}
i.e. the eikonal equation \eqref{eik} with respect to $x$ and fixed parameter $y$ is satisfied at $x_1$.
Proceeding we get a series $d^2_{10}, d^2_{20}, \cdots, d^2_{k0},\cdots $ of approximations of the form
\begin{equation}
d_{k0}^2(x,y)=d^ 2_{A^ {-1}(y)}(x,y)+\sum_{j=1}^kc_{j0}\Pi_{r=0}^j\Pi_{l=1}^n(x_l-x^r_l)^2d^ 2_{A^ {-1}(x^j)}(x,y).
\end{equation}
Having determined the real numbers $c_{10},\cdots c_{(k-1)0}$ we obtain the real number $c_{k0}$ by solving
\begin{equation}\label{eikatk}
d_{k0}^2(x_k,y)=\frac{1}{4}\sum_{ij} a_{ij}(x_k) d^ 2_{k0,x_i}(x_k,y)d^ 2_{k0,x_j}(x_k,y).
\end{equation}
for $c_{k0}$.
Continuing this procedure for $N$ interpolation points we get a polynomial of the form
\begin{equation}\label{dN1}
d_{N0}^2(x,y)=d^ 2_{A^ {-1}(y)}(x,y)+\sum_{j=1}^Nc_{j0}\Pi_{r=0}^j\Pi_{l=1}^n(x_l-x^r_l)^2d^ 2_{A^ {-1}(x^j)}(x,y).
\end{equation} 
with $N$ real numbers $c_{j0}$ obtained recursively by plugging in $d^2_{j0}$ with one degree of freedom $c_{j0}$ into \eqref{eikatk}. 

Analogous constructions are done to approximate $x\rightarrow d^2(x,x^j)$ for $k=1,\cdots ,N$ with
\begin{equation}\label{dNk}
d_{Nk}^2(x,x^k)=d^ 2_{A^ {-1}(y)}(x,x^k)+\sum_{j=1}^Nc_{jk}\Pi_{r=0}^j\Pi_{l=1}^n(x_l-x^r_l)^2d^ 2_{A^ {-1}(x^j)}(x,x^k),
\end{equation} 
with $c_{jk}$ computed analogously.
The construction of the functions $d^2_{N0},\cdots , d^2_{NN}$ suffices to approximate $d^2$ (we do not need to synthesize these functions into one function, for example by a Lagrangian polynomial interpolation). Note that for $j=0,\cdots N$ the function $d^2_{Nk}$ satisfies the equation
\begin{equation}\label{bdpk}
\begin{array}{ll}
d^2(x,x^k)=\frac{1}{4}\sum_{ij}a_{ij}(x)d^2_{x_i}(x,x^k)d^2_{x_j}(x,x^k)\\
\\
\mbox{ with boundary condition }\\
\\
d^2(x,x^k)=0~\mbox{ iff }~x=x^k.  
\end{array}
\end{equation}
at all interpolation points $x^0,\cdots x^N$ by construction.
\begin{rem}
Note that in the preceding construction no restrictions on the choice of the interpolation points are made.
This does not mean that one may search for an optimal choice of interpolation points and improve efficiency and convergence. We are free to choose a certain set of interpolation points (for example Chebyshev nodes).
But these are purely computational aspects which will be exploited elsewhere.
\end{rem}

\begin{rem}
Note that we have constructed an approximation of the squared metric $d^2$. The metric $d$ is then approximated naturally by the squareroot of the approximation of the squared metric, i.e. we consider the function
\begin{equation}
x\rightarrow d_{Nk}(x,x^k):=\sqrt{d^2_{Nk}(x,x^k)}
\end{equation}
to be the approximation of the metric function $x\rightarrow d(x,x^k)$.
\end{rem}

\subsection{Construction of $H^{s,p}$-approximations}

We refine the construction of the preceding section by construction of an approximation which solves not only \eqref{eik}, (or the set of equations \eqref{lf}, \eqref{lg} with boundary conditions \eqref{lh}), but also all multivariate derivatives of \eqref{eik} up to a given order $m$ at the interpolation points. It turns out then that these polynomials are $H^{s,p}$-approximations for $s\leq m$. The approximation is constructed recursively again. 
For a multiindex $\beta$ of order $|\beta|=m\geq 3$ we denote the approximations of order $ d^2_{M(\beta_m)^{n,N}}$ or just $d^2_{M(\beta_m)}$ if we do not want to refer to the number of interpolation points $N$ and the dimension of the problem $n$ explicitly. The choice of the mesh is free again (in principle). We just assume that a set $\left\lbrace x_1,\cdots ,x_N \right\rbrace$ of interpolation points is given. Again we may construct functions
$x\rightarrow d^2_{M(\beta)0}(x,y)$, $x\rightarrow d^2_{M(\beta)0}(x,x^1)$,..., and $x\rightarrow d^2_{M(\beta)0}(x,x^N)$. We shall construct the first function $x\rightarrow d^2_{M(\beta)0}(x,y)$ for arbitrary multiindex $\beta$. The other functions can be constructed completely analogously.
We start with the $L^p$-approximation.
\begin{equation}
d_{N0}^2(x,y)=d^ 2_{A^ {-1}(y)}(x,y)+\sum_{j=1}^Nc_{j0}\Pi_{r=1}^j\Pi_{l=1}^n(x_l-x^r_l)^2d^ 2_{A^ {-1}(x_j)}(x,y),
\end{equation} 
where the numbers $c_{j1}$ have been determined according to section 4.1..
Next we define $d^2_{M(\beta)0}(x,x^N)$ for multiindices of order $|\beta|=3$. Let $\beta^0,\cdots, \beta^k,\cdots ,\beta^{R}$ a list of multiindices of order $3$. The length $R$ of this list is dependent of the dimension $n$ of course. Start with $\beta^0=(\beta^0_1,\cdots,\beta^0_n)$ and let $\gamma^0$ be an multiindex with $|\gamma|=2$ such that $\beta^0-\gamma =1_i$ for some index $i$. Define (recall that $x^0=y$) 
\begin{equation}
d^2_{\beta^0 0}(x,y)=d^2_{N0}(x,y)+\frac{1}{\beta^0!}c_{\beta^0}^0(x-y)^{\beta^0}.
\end{equation}
Then plug $d^2_{\beta^0 0}(x,y)$ into the equation
\begin{equation}\label{beta0}
\begin{array}{ll}
\partial^{(\beta^0-\gamma^0)}_xd^2(x,y)=\partial^{(\beta^0-\gamma^0)}_x\left(\frac{1}{4}\sum_{ij}a_{ij}(x)\frac{\partial d^2}{\partial x_i}\frac{\partial d^2}{\partial x_j} \right),
\end{array}
\end{equation}
evaluate at $x=x^0=y$ and solve for the real number $c_{\beta^0}^0$. Then proceed recursively: having defined the function $x\rightarrow d^2_{\beta^0 (k-1)}(x,y)$ define
\begin{equation}
d^2_{\beta^0 k}(x,y)=d^2_{\beta^0 (k-1)}(x,y)+c_{\beta^0}^k\Pi_{l=0}^{k-1}(x-x^l)^{\beta^0+{\bf 1}}\frac{1}{\beta^0!}(x-x^k)^{\beta^0},
\end{equation}
where ${\bf 1}=(1,1,\cdots,1)$.
Then plug $d^2_{\beta^0 k}(x,y)$ into the equation \eqref{beta0}, evaluate at $y$, and solve for $c_{\beta^0}^k$. When $k=N$ we have got the approximation 
\begin{equation}
d^2_{\beta^0 N}(x,y)=d^2_{N0}(x,y)+\sum_{k=0}^Nc_{\beta^0}^k\Pi_{l=0}^{k-1}(x-x^l)^{\beta^0+{\bf 1}}\frac{1}{\beta^0!}(x-x^k)^{\beta^0}.
\end{equation}
with $N+1$ real numbers $c_{\beta^0}^k$ for $0\leq k\leq N$ determined recursively. Note that the function
$x\rightarrow d^2_{\beta^0 k}(x,y)$ satisfies the equations \eqref{eik} and \eqref{beta0} at all interpolation points $x^0,\cdots,x^N$. Then we take the next multiindex $\beta^1$ from the list of multiindices of order $3$ (i.e. $|\beta^1|=3$) where we may assume that $\beta^1-\gamma^1=1_k$ for some multiindex $\gamma^1$ with $|\gamma^1|=2$ and some index $k$. An analogous construction as in the case of $\beta^0$ can be done. The only difference is that we start with $d^2_{\beta^0 N}(x,y)$ instead of $d^2_{N0}(x,y)$. We get an approximation of the form
\begin{equation}
d^2_{\beta^1 N}(x,y)=d^2_{\beta^0 k}(x,y)+\sum_{k=0}^Nc_{\beta^1}^k\Pi_{l=0}^{k-1}(x-x^l)^{\beta^1+1}\frac{1}{\beta^1!}(x-x^k)^{\beta^1}.
\end{equation}
where the real numbers are computed recursively by plugging the current approximation into the equation 
\begin{equation}\label{beta0}
\begin{array}{ll}
\partial^{(\beta^1-\gamma^1)}_xd^2=\partial^{(\beta^1-\gamma^1)}_x\left(\frac{1}{4}\sum_{ij}a_{ij}(x)\frac{\partial d^2}{\partial x_i}\frac{\partial d^2}{\partial x_j} \right),
\end{array}
\end{equation}
evaluating at the current interpolation point and solving for the currently undetermined real number $c_{\beta^1}^k$. 
Doing this for all the multiindices of order $3$ in the list above we get the approximation
\begin{equation}
d^2_{M(\beta_3)}(x,y):=d^2_{\beta^N N}(x,y).
\end{equation}
Note that by construction the function $x\rightarrow d^2_{M(\beta_3)}(x,y)$ satisfies the equation \eqref{eik} and all its first order derivative equations
\begin{equation}
\partial^{i}_xd^2=\frac{1}{4}\partial^{i}_x\left( \sum_{lm}a_{lm}(x)d^2_{x_l}d^2_{x_m}\right) ,~~1\leq i\leq n,
\end{equation}
at all interpolation points $x^0=y,x^1,\cdots, x^N$. This completes the stage of construction for multiindices of order $3$. Next assume that the construction for the approximation 
\begin{equation}
x\rightarrow d^2_{M(\beta_m)}(x,y)
\end{equation}
of order $m$ has been completed. Then we may list the multiindices of order $m+1$, i.e. consider a list of multiindices $\delta^0,\delta^1,\cdots, \delta^{R_{m+1}}$ such that $|\delta|=m+1$. The procedure is then quite similar as in the stage for multiindices of order $3$. Therefore we give a very short description. Starting with the multiindex $\delta^0$ there is a multiindex $\beta^k$ of order $m$ (i.e. $|\beta^k|=m$) such that $\delta^0-\beta^k=1_i$ for some index $i$. Then we get successive approximations
\begin{equation}
d^2_{\delta^0 k}(x,y)=d^2_{M(\beta_m)}(x,y)+\sum_{r=0}^{k}c_{\delta^0}^r\Pi_{l=0}^{r-1}(x-x^l)^{\delta^0+{\bf 1}}\frac{1}{\delta^0!}(x-x^r)^{\delta^0},
\end{equation}
where the real numbers $c_{\delta^0}^k$ are succesively determined by plugging in the function $x\rightarrow d^2_{\delta^0 k}(x,y)$ into the equation 
\begin{equation}
\partial^ {\beta^k}d^2=\frac{1}{4}\partial^ {\beta^k}\left( \sum_{lm}a_{lm}(x)d^2_{x_l}d^2_{x_m}\right) ,
\end{equation}
evaluated at the interpolation point $x^k$ (Note that $\partial^{\beta^k}=\partial^ {\delta^0-1_i}$).
After $N+1$ steps we get the approximation function $x\rightarrow d^2_{\delta^0 N}(x,y)$. Having defined
$x\rightarrow d^2_{\delta^l N}(x,y)$ for $l=0,\cdots p-1$ the next multiindex $\delta^r$ may be such that there is an multiindex $\beta^h$ of order $m$ such that $\delta^{r}-\beta^h=1_i$ for some index $i$. We may then define $x\rightarrow d^2_{\delta^p k}(x,y)$
\begin{equation}
d^2_{\delta^p k}(x,y)=d^2_{M(\beta_m)}(x,y)+\sum_{r=0}^{k}c_{\delta^p}^r\Pi_{l=0}^{r-1}(x-x^l)^{\delta^p+{\bf 1}}\frac{1}{\delta^p!}(x-x^r)^{\delta^p},
\end{equation}
and determine the constants $c_{\delta^p}^r$ by plugging in the function $x\rightarrow d^2_{\delta^r k}(x,y)$ into the equation 
\begin{equation}
\partial^ {\beta^h}d^2=\frac{1}{4}\partial^ {\beta^h}\left( \sum_{lm}a_{lm}(x)d^2_{x_l}d^2_{x_m}\right) ,
\end{equation}
and evaluate at $x^k$.
Finally, we get the approximation of order $m+1$, namely 
\begin{equation}\label{betam}
d^2_{M(\beta_{m+1})}=d^2_{\delta^{R_{m+1}} N}(x,y).
\end{equation}
Note that this approximation satisfies the eikonal equation \eqref{eik} and all its derivatives
up to order $m+1$, i.e. all equations
\begin{equation}
\partial^{\alpha}_x d^2=\frac{1}{4}\partial^{\alpha}_x \left( \sum_{lm}a_{lm}(x)d^2_{x_l}d^2_{x_m}\right) 
\end{equation}
with $|\alpha|\leq m+1$ at all interpolation points $x^1,\cdots x^N$.
\begin{rem} 
Note that at some stage of the construction we may have a multiindex $\gamma$ such that $\gamma-\alpha=1_{i_0}$ for some $\alpha$ and some index $i_0$. Then  the terms in the $\alpha$th derivative of the eikonal equation evaluated at $x^k$ that do not annihilate a term of form $c_{\gamma}^k\Pi_{l=0}^{k-1}(x-x^l)^{\gamma+{\bf 1}}(x-x^k)^{\gamma}$ are quite easily computed. For this reason the constants of the form $c_{\gamma}^k$ are quite easily computed. You can see very easily this by writing the $\alpha$th derivative of the eikonal equation invoking symmetry $a_{ij}=a_{ji}$. We have 
\begin{equation}\label{alpha}
\begin{array}{ll}
\partial^{\alpha}d^2(x,y)=\partial^{\alpha}\left(\frac{1}{4}\sum_{ij}a_{ij}(x)\frac{\partial d^2}{\partial x_i}\frac{\partial d^2}{\partial x_j} \right)\\
\\
=\frac{1}{2}\sum_{ij}a_{ij}(x)\left( \partial^{\alpha}\frac{\partial d^2}{\partial x_i}\right) \frac{\partial d^2}{\partial x_j} +\frac{1}{4}\sum_{ij}\left(\frac{\partial^{\alpha}}{\partial x^{\alpha}} a_{ij}(x)\right) \frac{\partial d^2}{\partial x_i} \frac{\partial d^2}{\partial x_j}\\
\\
 +\frac{1}{4}\sum_{ij}\sum_{\beta <\alpha}\sum_{\gamma \leq\beta}\binom{\alpha}{\beta}\binom{\beta}{\gamma}
\left( \partial^{\beta}a_{ij}(x)\right) 
 \left( \partial^{\alpha -\beta-\gamma}\frac{\partial d^2}{\partial x_i}\right)  \partial^{\gamma} \frac{\partial d^2}{\partial x_j} 
\end{array}
\end{equation}
If the indicated approximation is plugged into \eqref{alpha} and evaluated at $x^k$ only the terms $\frac{1}{2}\sum_{j}a_{i_0j}(x)\left( \partial^{\alpha}\frac{\partial d^2}{\partial x_{i_0}}\right) \frac{\partial d^2}{\partial x_j}$ (evaluated for approximations $d^2_{\gamma k}$ at interpolation point $x^k$) do not annihilate terms of  form $c_{\gamma}^k\Pi_{l=0}^{k-1}(x-x^l)^{\gamma+{\bf 1}}(x-x^k)^{\gamma}$.
\end{rem}

\section{Error estimates for the regular polynomial interpolation algorithm}
We first consider error estimates for $L^p$-approximations, and then extend our estimates to $H^{s,p}$-approximations. In the whole Section we consider a bounded domain $\Omega \subseteq {\mathbb R}^n$ and assume that the coefficient functions $a_{ij}$ are $C^{\infty}$.
\subsection{Error estimates for $L^p$ approximation}

We have
\begin{thm}
The approximations $d^2_{Nk}$ defined in \eqref{dNk} are $L^p$- approximations of the boundary value problems
of form \eqref{bdpk}, i.e. $L^p$- approximations for functions of form $x\rightarrow d^2(x,x^k)$ for $p>1$.
\end{thm}
\begin{proof}
Let $x$ and $y$ be two points connected by a geodesic curve $\gamma$ given in local coordinates with values in ${\mathbb R}^n$. Let us assume also that $x$ and $y$ are interpolation points. We have no solution for the curve $\gamma$ in general, but there are lets say $k$ points $z^0=x,z^1 \cdots z^k=y$ in the image of the curve $\gamma$ with Euclidean distance less than a certain mesh size $h$. Clearly,
\begin{equation}
d(x,y)=\sum_{i=0}^N d(z^i,z^{i+1}) 
\end{equation}
Next define an approximative distance along the geodesic of form
\begin{equation}
d_g(x,y)=\sum_{i=0}^n d_g(z^i,z^{i+1}),
\end{equation}
where $d_g$ is the squareroot of $d_g^2(z^{i},z^{i+1}):=\sum_{lm}a^{lm} (z^i_m-z^{i+1}_m)(z^i_l-z^{i+1}_l)$.
Since $y$ is fixed $d$ is approximated by $d_{N0}$ and we estimate
\begin{equation}\label{metest}
d(x,y)-d_{N0}(x,y)=d(x,y)-d_g(x,y)+d_g(x,y)-d_{N0}(x,y)
\end{equation}
Our analysis showed that the local approximation of $d^2$ by $d_g^2$ is of order $O(h^3)$ hence the approximation of $d$ by $d_g$ is of order $O\left( h^{\frac{3}{2}}\right)$, hence with generic constant $C$ we have for the first summand on the right hand side of \eqref{metest} 
\begin{equation}
|d(x,y)-d_g(x,y)|=\sum_{i=0}^N |\left( d(z^i,z^{i+1})-d_g(z^i,z^{i+1})\right)|\leq C\sqrt{h}  
\end{equation}
The modulus of the first summand on the right hand side can be estimated by
\begin{equation}
|d(x,y)-d^g(x,y)|\leq C\sqrt{h} 
\end{equation}
Since $\Omega$ is a compact bounded domain, the $C^{\infty}$ coefficient functions  $a^{ij}$ are  Lipschitz Only locally Lipschitz is needed). Assuming a suitable choice of the points on the geodesic for the second summand we get by an elementary argument that
\begin{equation}
\|d_g(x,y)-d_{N0}(x,y)\|_{L^p}\leq \sum_{i=0}^N \|d(z^i,z^{i+1})-d_{g}(z^i,z^{i+1})\|_{L^p}
\leq Ch^{p-1}.
\end{equation}

\end{proof}

\subsection{Error estimates for $H^{s,p}$ approximation}

\begin{thm}
The approximations $d^2_{M(\beta_m)}$ defined in \eqref{betam} are $H^{s,p}$- approximations of the boundary value problems
of form \eqref{bdpk} for $s\leq m$, i.e. $H^{s,p}$- approximations for functions of form $x\rightarrow d^2(x,x^k)$ for $p>1$.
\end{thm}

\begin{proof}
For fixed $y$ the function $x\rightarrow d^2(x,y)$ and the function $x\rightarrow d^2_{M(\beta_m)}(x,y)$ both satisfy the eikonal equation and its derivatives at any interpolation point by construction. That means that for all interpolation points $x_j,~1\leq j\leq N$ and all derivatives $\gamma\leq m$ we have
\begin{equation}
\partial^{\gamma}_xd^2(x_j,y)=\partial^{\gamma}_x d^2_{M(\beta_m)}(x_j,y).
\end{equation}
Next recall a multivariate version of Taylor's theorem
\begin{thm}
If $f\in C^{\infty}$, then for all positive integers $M$ we have
\begin{equation}
\begin{array}{ll}
f(x+y)=\sum_{|\alpha|<M}\frac{(\partial{\alpha}f)(x)}{\alpha!}y^{\alpha}\\
\\
+M\sum_{|\gamma| =M}\frac{y^{\gamma}}{\gamma !}\int_0^1(1-\theta)^{M-1}(\partial^{\gamma}f)(x+\theta y)d\theta
\end{array}
\end{equation}
\end{thm}
Applying this formula, we see from our construction of $x\rightarrow d^2_{M(\beta_m)}(x,y)$ that the local order of approximation of $x\rightarrow d^2(x,y)$ is $O(h^{3+m})$. A similar reasoning as in the preceding Section leads to the result. Note here that the same reasoning holds when $y$ is replaced by another interpolation point $x^j$.
\end{proof}

\begin{rem}(Sharper error estimates)
A little analysis shows that the local order of approximation is
\begin{equation}
d^2(x,y)-d^2_{M(\beta_m)}(x,y)\leq C
P\frac{h^{m}}{m !},
\end{equation}
where $P$ is the number of multiindices of order $m$ and
\begin{equation}
C:=2\max\left\lbrace \sum_{|\gamma| =M}\sup_{x\in \Omega}\partial^{\gamma}d^2(x,y),\sum_{|\gamma| =M}\sup_{x\in \Omega}\partial^{\gamma}d^2_{M(\beta)}(x,y)\right\rbrace 
\end{equation}
Similarly for $|\beta | \leq m $ we get
\begin{equation}
\partial^{\beta}d^2(x,y)-\partial^{\beta}d^2_{M(\beta_m}(x,y)\leq C
\frac{h^{m -|\beta|}}{(m -|\beta |) !},
\end{equation}
Since we are working on a bounded domain and $d^2$ is $C^{\infty}$ there is some bound $C$, but not a priori known. However bounds for $C$ may be obtained from a priori estimates by inspection of the eikonal equation.
It is clear that
\begin{equation}
 d^2=\frac{1}{4}\sum_{ij}a_{ij}d^2_{x_i}d^2_{x_j}.
\end{equation}
is equivalent to
\begin{equation}
 d^2=\frac{1}{4}\sum_{i}\lambda_{i}d^2_{x_i}d^2_{x_i},
\end{equation}
and, hence
\begin{equation}
|d^2_{x_i}|\leq \frac{4d^2}{\lambda_{\min}},
\end{equation}
where $\lambda_{\min}=\min_i \inf_{x\in \Omega}\lambda(x)$.
Further a priori estimates for the derivatives may be obtained from derivatives of the eikonal equation.
\end{rem}

\section{Analytic approximations of the fundamental solution of parabolic equations}

Consider the parabolic equation
\begin{equation}
\frac{\partial p}{\partial t}-\frac{1}{2}\sum_{ij}a_{ij}(x)\frac{\partial^2 p}{\partial x_i\partial x_j}-\sum_i b_i(x)\frac{\partial p}{\partial x_i}=0
\end{equation}
on some domain $(0,T)\times \Omega \subseteq (0,T)\times {\mathbb R}^n$, and where 
$x\rightarrow (a_{ij}(x))$ is a matrix-valued $C^{\infty}$-function with symmetric positive matrix $(a_{ij}(x))$ 
for all $x\in \Omega$, and $x\rightarrow b(x)$ is also a $C^{\infty}$-function. 
Plugging in the Ansatz
\begin{equation}\label{WKBrep}
p(t,x,y)=\frac{1}{\sqrt{2\pi t}^n}\exp\left(-\frac{d^2(x,y)}{2 t}+\sum_{k= 0}^{\infty}c_k(x,y)t^k\right), 
\end{equation}
leads to the recursive equation \eqref{eik} for $d^2$. Given $d^2$ the first order recursive equation 
\begin{equation}\label{c01e}
-\frac{n}{2}+\frac{1}{2}Ld^2+\frac{1}{2}\sum_{i} \left( \sum_j\left( a_{ij}(x)+a_{ji}(x)\right) \frac{d^2_{x_j}}{2}\right) \frac{\partial c_{0}}{\partial x_i}(x,y)=0,
\end{equation}
together with the boundary condition 
\begin{equation}\label{c01b}
c_0(y,y)=-\frac{1}{2}\ln \sqrt{\mbox{det}\left(a^{ij}(y) \right) }
\end{equation}
determines $c_0$ uniquely for each $y\in {\mathbb R}^n$. Furthermore, having computed all WKB-coefficient functions $c_l$ up to order $k$, for $k+1\geq 1$ the coefficient function $c_{k+1}$  can be computed via the first order equation
\begin{equation}\label{1gaa}
\begin{array}{ll}
(k+1)c_{k+1}(x,y)+\frac{1}{2}\sum_{ij} a_{ij}(x)\Big(
\frac{d^2_{x_i}}{2}\frac{\partial c_{k+1}}{\partial x_j}
+\frac{d^2_{x_j}}{2} \frac{\partial c_{k+1}}{\partial x_i}\Big)\\
\\
=\frac{1}{2}\sum_{ij}a_{ij}(x)\sum_{l=0}^{k}\frac{\partial c_l}{\partial x_i} \frac{\partial c_{k-l}}{\partial x_j}
+\frac{1}{2}\sum_{ij}a_{ij}(x)\frac{\partial^2 c_k}{\partial x_i\partial x_j}  
+\sum_i b_i(x)\frac{\partial c_{k}}{\partial x_i},
\end{array}
\end{equation}
with the boundary conditions
\begin{equation}\label{Rk}
c_{k+1}(x,y)=R_k(y,y) \mbox{ if }~~x=y,
\end{equation}
$R_k$ being the right side of \eqref{1gaa}. 
We will show in a subsequent paper that equations \eqref{c01e}
\eqref{c01b}, and \eqref{1gaa},\eqref{Rk} can be solved or approximated to higher order if \eqref{eik} is solved or approximated to higher order. We have
\begin{prop}
In order to compute the WKB-approximation up to order $k$ a $H^{s,p}$ approximation of $d^2$ for $s\geq 2k$ is sufficient.  
\end{prop}
\begin{proof}
In each recursion step \eqref{1gaa} an operator of order $2$ is applied to the previously computed WKB-coefficients. 
\end{proof}

\section{Conclusion and final remarks on computational issues}

We have established a stable algorithm for efficient computation of the length of geodesics as a function of two arbitrary points on $C^k$- Riemannian manifold with minimal geodesic as well as of  partial derivatives of the length functional (of principally any order) accurately. We established error estimates in arbitrary Sobolev norms. We showed how the algorithm can be applied in order to compute fundamental solutions of irreducible linear parabolic equations. There are many obvious applications to mathematical physics and finance as well as  to statistics, e.g. to the maximum log-likelihood method, to option pricing, computing transition amplitudes etc.. Finally, we remark that the interpolation polynomials should not be evaluated in the way they are constructed. Here careful implementation of Horner schemes is needed. But this computational issues will be considered in a subsequent paper.

\end{document}